\documentclass[10pt,leqno]{amsart}
\theoremstyle{definition}
\usepackage{indentfirst, amssymb,amsmath,amsthm}
\usepackage{csquotes}
\usepackage{hyperref}
\newtheorem*{theoA}{Theorem A}
\newtheorem*{theoB}{Theorem B}
\newtheorem*{theoC}{Theorem C}

\newtheorem{theo}{Theorem}[section]
\newtheorem{lem}{Lemma}[section]
\newtheorem{cor}{Corollary}[section]

\newtheorem{defi}{Definition}[section]
\newtheorem{rem}{Remark}[section]
\newtheorem{question}{Question}[section]
\newcommand{\ol}{\overline}
\newcommand{\be}{\begin{equation}}
\newcommand{\ee}{\end{equation}}
\newcommand{\beas}{\begin{eqnarray*}}
\newcommand{\eeas}{\end{eqnarray*}}
\newcommand{\bea}{\begin{eqnarray}}
\newcommand{\eea}{\end{eqnarray}}
\newcommand{\lra}{\longrightarrow}
\numberwithin{equation}{section}
\begin{document}
\title[Some inequalities related to Differential Monomials]{Some inequalities related to Differential Monomials}
\date{}
\author[B. Chakraborty ]{Bikash Chakraborty }
\date{}
\address{Department of Mathematics, Ramakrishna Mission Vivekananda Centenary College, Rahara,
West Bengal 700 118, India.}
\email{bikashchakraborty.math@yahoo.com, bikashchakrabortyy@gmail.com}
\maketitle
\let\thefootnote\relax
\footnotetext{2010 Mathematics Subject Classification: 30D30, 30D20, 30D35.}
\footnotetext{Key words and phrases : Transcendental Meromorphic function, Differential Monomials.}
\begin{abstract} The aim of this paper is to consider the value distribution of a differential monomial generated by a transcendental meromorphic function. \end{abstract}
\section{Introduction}
In this article, we use the standard notations of value distribution theory (see, Hayman's Monograph (\cite{8})). It will be convenient to let $E$ denote any set of positive real numbers of finite linear (Lebesgue) measure, not necessarily the same at each occurrence. For any non-constant meromorphic function $f$, we denote by $S(r,f)$ any quantity satisfying $$S(r, f) = o(T(r, f))~~\text{as}~~r\to\infty,~r\not\in E.$$
In addition, in this paper, we also use another type of notation $S^{*}(r,f)$ which is defined as
$$S^{*}(r,f)=o(T(r,f))~~\text{as}~~r\to\infty,~r\not\in E^{*},$$
where $E^{*}$ is a set of logarithmic density $0$.\par
By small function with respect to a non-constant meromorphic function $f$, we mean a meromorphic function $b=b(z)(\not\equiv 0,\infty)$ which satisfies that $T(r,b)=S(r,f)$ as $r\lra \infty, r\not\in E$.\par
Throughout this paper, we always assume that $f$ is a transcendental meromorphic function in the complex plane $\mathbb{C}$. \par
In 1979, Mues (\cite{m}) proved that for a transcendental meromorphic function $f(z)$ in $\mathbb{C}$, $f^{2}f'-1$ has infinitely many zeros. In 1992, Q. Zhang (\cite{qz}) proved the quantitative version of Mues's Result  as follows:
\begin{theoA} For a transcendental meromorphic function $f$, the following inequality holds :
$$T(r,f)\leq 6N\bigg(r,\frac{1}{f^{2}f'-1}\bigg)+S(r,f).$$
\end{theoA}
In this direction  Huang and Gu (\cite{hg}) obtained the following result:
\begin{theoB} Let $f$ be a transcendental meromorphic function and $k$ be a positive integer. Then
$$T(r,f)\leq 6N\bigg(r,\frac{1}{f^{2}f^{(k)}-1}\bigg)+S(r,f).$$
\end{theoB}
In this connection, one can easily see that the following result is an immediate corollary of \emph{Theorem 3.2} of Lahiri and Dewan (\cite{ld}).
\begin{theoC} Let $f$ be a transcendental meromorphic function and $a$ be a non zero complex constant. Let $l\geq3$, $n\geq1$, $k\geq1$ be positive integers. Then
$$T(r,f)\leq \frac{1}{l-2}\ol{N}\bigg(r,\frac{1}{f^{l}(f^{(k)})^{n}-a}\bigg)+S(r,f).$$
\end{theoC}
Next we introduce the following definition:
\begin{defi} Let $q_{1},q_{2},...,q_{k}$ be $k(\geq1)$ non-negative integers and $a$ be a non zero complex constant. Then the expression defined by $$M[f]=a (f)^{q_{0}}(f')^{q_{1}}...(f^{(k)})^{q_{k}}$$ is known as differential monomial generated by $f$. Next we define $\mu=q_{0}+q_{1}+...+q_{k}$ and $\mu_{*}=q_{1}+2q_{2}+...+kq_{k}$. In literature, the terms $\mu$ and $\mu+\mu_{*}$ are known as the degree and weight of the differential monomial respectively. \par Here, in our paper, we always take $q_{0}\geq1,~q_{k}\geq1$.
\end{defi}
Since differential monomial $M[f]$ is the general form of $(f)^{q_{0}}(f^{(k)})^{q_{k}}$, so from the above discussion it is natural to ask the following questions:
\begin{question} Are there any positive constants $B_{1},B_{2}>0$ such that following hold?
\begin{enumerate}
\item [i)] $T(r,f)\leq B_{1}~ N\bigg(r,\frac{1}{M[f]-c}\bigg)+S(r,f),$
\item [ii)] $T(r,f)\leq B_{2}~\ol{N}\bigg(r,\frac{1}{M[f]-c}\bigg)+S(r,f),$
\end{enumerate}
where $M[f]$ is a differential monomial generated by a non constant transcendental meromorphic function $f$ and $c$ is any non zero constant.
\end{question}
To answer the above questions are the motivations of this paper. Before going to our main results we first explain some notations and definitions:
\begin{defi}
Let $k$ be a positive integer, for any constant $a$ in the complex plane. We denote
\begin{enumerate}
\item [i)] by $N_{k)}(r,\frac{1}{( f -a)})$ the counting function of $a$-points of $f$ with multiplicity $\leq k$,
\item [ii)] by $N_{(k}(r,\frac{1}{( f -a)})$ the counting function of $a$-points of $f$ with multiplicity $\geq k$.
\end{enumerate}
Similarly, the reduced counting functions $\ol{N}_{k)}(r,\frac{1}{( f -a)})$ and $\ol{N}_{(k}(r,\frac{1}{( f -a)})$ are defined.
\end{defi}
\section{Main Results}
\begin{theo}\label{th1} Let $f$ be a transcendental meromorphic function and $k\geq2$, $q_{0}\geq2$, $q_{i}\geq0~(i=1,2,..,k-1)$, $q_{k}\geq2$ be integers.
Then
\bea\label{eq0.1} T(r,f)\leq \frac{1}{q_{0}-1}N\bigg(r,\frac{1}{M[f]-1}\bigg)+S^{*}(r,f),\eea
where $S^{*}(r,f)=o(T(r,f))$ as $r\to\infty,r\not\in E$, $E$ is a set of logarithmic density $0$.
\end{theo}
\begin{cor} Let $f$ be a transcendental meromorphic function and $k\geq2$, $q_{0}\geq2$, $q_{i}\geq0~(i=1,2,..,k-1)$, $q_{k}\geq2$ be integers. For a no zero complex constant $\alpha$, we have
\bea\label{eq0.2} T(r,f)\leq \frac{1}{q_{0}-1}N\bigg(r,\frac{1}{(f)^{q_{0}}(f')^{q_{1}}...(f^{(k)})^{q_{k}}-\alpha}\bigg)+S^{*}(r,f).\eea
\end{cor}
\begin{cor} Let $f$ be a transcendental meromorphic function and $k\geq2$, $q_{0}\geq2$, $q_{i}\geq0~(i=1,2,..,k-1)$, $q_{k}\geq2$ be integers. Then $(f)^{q_{0}}(f')^{q_{1}}...(f^{(k)})^{q_{k}}$ assumes every non-zero finite value infinitely often.
\end{cor}
\begin{theo}\label{th2} Let $f$ be a transcendental meromorphic function and $k\geq1$, $\mu-\mu_{*}\geq3$, $q_{0}\geq1$, $q_{i}\geq0~(i=1,2,..,k-1)$, $q_{k}\geq1$ be integers. Then
\bea\label{eq0.3} T(r,f)\leq \frac{1}{\mu-\mu^{*}-2}\ol{N}\bigg(r,\frac{1}{M[f]-1}\bigg)+S(r,f),\eea
where $S(r,f)=o(T(r,f))$ as $r\to\infty,r\not\in E$, $E$ is a set of finite linear measure.
\end{theo}
\begin{cor} Let $f$ be a transcendental meromorphic function and $k\geq1$, $\mu-\mu_{*}\geq3$, $q_{0}\geq1$, $q_{i}\geq0~(i=1,2,..,k-1)$, $q_{k}\geq1$ be integers. For a no zero complex constant $\alpha$, we have
\bea\label{eq0.4} T(r,f)\leq \frac{1}{\mu-\mu^{*}-2}\ol{N}\bigg(r,\frac{1}{(f)^{q_{0}}(f')^{q_{1}}...(f^{(k)})^{q_{k}}-\alpha}\bigg)+S(r,f).\eea
\end{cor}
\begin{cor} Let $f$ be a transcendental meromorphic function and $k\geq1$, $\mu-\mu_{*}\geq3$, $q_{0}\geq1$, $q_{i}\geq0~(i=1,2,..,k-1)$, $q_{k}\geq1$ be integers. Then $(f)^{q_{0}}(f')^{q_{1}}...(f^{(k)})^{q_{k}}$ assumes every non-zero finite value infinitely often.
\end{cor}
\begin{theo}\label{th3}Let $f$ be a transcendental meromorphic function and $k\geq1$, $\mu-\mu_{*}\geq 5-q_{0}$, $q_{0}\geq1$, $q_{i}\geq0~(i=1,2,..,k-1)$, $q_{k}\geq1$ be integers.
Then
\bea\label{eq0.5} T(r,f)\leq \frac{1}{\mu-\mu^{*}-4+q_{0}}\ol{N}\bigg(r,\frac{1}{M[f]-1}\bigg)+S(r,f).\eea
\end{theo}
\begin{cor} Let $f$ be a transcendental meromorphic function and $k\geq1$, $\mu-\mu_{*}\geq5-q_{0}$, $q_{0}\geq1$, $q_{i}\geq0~(i=1,2,..,k-1)$, $q_{k}\geq1$ be integers. For a no zero complex constant $\alpha$, we have
\bea\label{eq0.6} T(r,f)\leq \frac{1}{\mu-\mu^{*}-4+q_{0}}\ol{N}\bigg(r,\frac{1}{(f)^{q_{0}}(f')^{q_{1}}...(f^{(k)})^{q_{k}}-\alpha}\bigg)+S(r,f).\eea
\end{cor}
\begin{cor} Let $f$ be a transcendental meromorphic function and $k\geq1$, $\mu-\mu_{*}\geq5-q_{0}$, $q_{0}\geq1$, $q_{i}\geq0~(i=1,2,..,k-1)$, $q_{k}\geq1$ be integers. Then $(f)^{q_{0}}(f')^{q_{1}}...(f^{(k)})^{q_{k}}$ assumes every non-zero finite value infinitely often.
\end{cor}
\section{Lemmas}
Let $a$ be a non zero complex constant and $q_{1},q_{2},...,q_{k}$ be $k(\geq1)$ non-negative integers. Define $\mu=q_{0}+q_{1}+...+q_{k}$ and $\mu_{*}=q_{1}+2q_{2}+...+kq_{k}$.\par
Let $M[f]=a (f)^{q_{0}}(f')^{q_{1}}...(f^{(k)})^{q_{k}}$ be a differential monomial generated by a non constant transcendental meromorphic function $f$ where we take $q_{0}\geq1,~q_{k}\geq1$.
\begin{lem}\label{lem1} For a non constant meromorphic function $g$,
$$N(r,\frac{g'}{g})-N(r,\frac{g}{g'})=\ol{N}(r,g)+N(r,\frac{1}{g})-N(r,\frac{1}{g'}).$$
\end{lem}
\begin{proof} For the proof, one go through the technique of formula (12) of (\cite{jh}).
\end{proof}
Now the following Lemma which plays the major role to prove Theorem\ref{th1} is a immediate corollary of Yamanoi's Celebrated Theorem(\cite{yam}). Yamanoi's Theorem is a correspondent result to the famous Gol'dberg Conjecture.
\begin{lem}(\cite{yam})\label{lem1.1} Let $f$ be a transcendental meromorphic function in $\mathbb{C}$ and let $k\geq2$ be an integer. Then
$$(k-1)\ol{N}(r,f)\leq N(r,\frac{1}{f^{(k)}})+S^{*}(r,f),$$
where $S^{*}(r,f)=o(T(r,f))$ as $r\to\infty,r\not\in E$, $E$ is a set of logarithmic density $0$.
\end{lem}
\begin{lem}\label{lem2} For any small function $b=b(z)(\not\equiv0,\infty)$ of $f$, $b(z)M[f]$ can not be a constant.
\end{lem}
\begin{proof} On contrary, let us assume \bea\label{kabir1} b(z)M[f]\equiv C,\eea for some constant $C$.\\
As $f$ is non  constant transcendental meromorphic function, so $C\not=0$.\par
Thus from (\ref{kabir1}) and Lemma of logarithmic derivative, it is clear that \bea\label{kabir2} m(r,\frac{1}{f})=S(r,f).\eea
Also it is clear from that (\ref{kabir1}) that \bea\label{kabir3} N(r,0;f)\leq N(r,0;M[f])=S(r,f).\eea
Thus $T(r,f)=S(r,f)$, which is absurd as $f$ is transcendental.
\end{proof}
\begin{lem}\label{lem3}
Let $f$ be a transcendental meromorphic function, then $T\bigg(r,b(z)M[f]\bigg)=O(T(r,f))$ and $S\bigg(r,b(z)M[f]\bigg)=S(r,f)$.
\end{lem}
\begin{proof} The proof is similar to the proof of the Lemma 2.4 of (\cite{ly}).
\end{proof}
\begin{lem}\label{lem4}Let $f$ be a transcendental meromorphic function.  Then
\bea\label{p1}\mu T(r,f) &\leq& \mu N(r,\frac{1}{f})+ \ol{N}(r,\infty;f)+N(r,\frac{1}{b M[f]-1})-N(r,\frac{1}{(b M[f])'})+S(r,f).\eea
\end{lem}
\begin{proof} As Lemma \ref{lem2} yields that $b(z)M[f]\not\equiv$ constant, so we can write
 $$\frac{1}{f^{\mu}}=\frac{b M[f]}{f^{\mu}}-\frac{(b M[f])'}{f^{\mu}}\frac{(b M[f]-1)}{(b M[f])'}.$$
Thus in view of Lemma \ref{lem3}, First Fundamental Theorem and Lemma \ref{lem1} we have
\bea\label{kabir5} \mu m(r,\frac{1}{f}) &\leq& m(r,\frac{b M[f]}{f^{\mu}})+m(r,\frac{(b M[f])'}{f^{\mu}})+m(r,\frac{b M[f]-1}{(b M[f])'})+O(1)\\
\nonumber&\leq& 2m(r,\frac{b M[f]}{f^{\mu}})+m(r,\frac{(b M[f])'}{b M[f]})+m(r,\frac{b M[f]-1}{(b M[f])'})+O(1)\\
\nonumber&\leq& T(r,\frac{b M[f]-1}{(b M[f])'})-N(r,\frac{b M[f]-1}{(b M[f])'})+S(r,f)\\
\nonumber&\leq& \ol{N}(r,\infty;f)+N(r,\frac{1}{b M[f]-1})-N(r,\frac{1}{(b M[f])'})+S(r,f)\eea
Thus \beas \mu T(r,f) &\leq& \mu N(r,\frac{1}{f})+ \ol{N}(r,\infty;f)+N(r,\frac{1}{b M[f]-1})-N(r,\frac{1}{(b M[f])'})+S(r,f).\eeas
\end{proof}
\begin{lem}\label{lem5}Let $f$ be a transcendental meromorphic function. If $q_{0}\geq1$, $q_{k}\geq 1$, then
\bea\label{kabir4} \mu T(r,f) &\leq& \ol{N}(r,\infty;f)+\ol{N}(r,0;f)+\mu_{*}\ol{N}_{(k+1}(r,0;f)+(\mu-q_{0})N_{k)}(r,0;f)\\
\nonumber && +\ol{N}(r,\frac{1}{M[f]-1})-N_{0}(r,\frac{1}{M[f]'})+S(r,f),\eea
where $N_{0}(r,\frac{1}{ (M[f])'})$ is the counting function of the zeros of $(M[f])'$ but not the zeros of $f(M[f]-1)$.
\end{lem}
\begin{proof} Clearly \bea\label{eq1} && \mu N(r,\frac{1}{f})+N(r,\frac{1}{M[f]-1})-N(r,\frac{1}{ (M[f])'})\\
\nonumber&\leq& \mu N(r,\frac{1}{f})-N_{\star}(r,\frac{1}{(M[f])'})+\ol{N}(r,\frac{1}{ M[f]-1})-N_{0}(r,\frac{1}{(M[f])'}),
\eea
where $N_{\star}(r,\frac{1}{(M[f])'})$ is the counting function of the zeros of $(M[f])'$ which comes from the zeros of $f$.\par
Let $z_{0}$ be a zero of $f$ with multiplicity $q$.\\
\textbf{Case-1} $q\geq k+1$.\\
It is easy to observe that \beas q\mu-\mu_{*}&\geq& k\mu-\mu_{*}+\mu\\ &\geq& \sum_{i=0}^{k-1}(k-i)q_{i}+\mu\\&\geq&(k+1)q_{0}+q_{k}\\&\geq&3.\eeas
Then $z_{0}$ is the zero of $(M[f])'$ of order atleast $q\mu-\mu_{*}-1$.\\
\textbf{Case-2} $q\leq k$.\\
Then $z_{0}$ is the zero of $(M[f])'$ of order atleast $qq_{0}-1$. Thus
\bea\label{eq2} &&\mu N(r,\frac{1}{f})-N_{\star}(r,\frac{1}{(M[f])'})\\
\nonumber&\leq& (\mu-q_{0}) N_{k)}(r,\frac{1}{f})+\ol{N}_{k)}(r,\frac{1}{f})+(\mu_{*}+1)\ol{N}_{(k+1}(r,\frac{1}{f})\eea
Now the proof follows from the Lemma \ref{lem4} and the inequalities (\ref{eq1}),(\ref{eq2}).
\end{proof}
\section {Proof of the Theorems}
\begin{proof} [\textbf{Proof of Theorem \ref{th1} }]
It is given that, $f$  is a transcendental meromorphic function and $k\geq2,~q_{0}\geq2,~q_{k}\geq2$. It is clear that $$(q_{0}-1)N(r,0;f)+(q_{k}-1)N(r,0;f^{(k)})\leq N(r,0;(M[f])').$$
Now in view of Lemma \ref{lem1.1} and Lemma \ref{lem4}, we have
\beas && \mu T(r,f)\\
\nonumber &\leq& (\mu-q_{0}+1) N(r,\frac{1}{f})+ (1-(k-1)(q_{k}-1))\ol{N}(r,\infty;f)+N(r,\frac{1}{M[f]-1})\\
\nonumber&&+S(r,f)+S^{*}(r,f)\\
\nonumber &\leq& (\mu-q_{0}+1) N(r,\frac{1}{f})+N(r,\frac{1}{M[f]-1})+S^{*}(r,f).
\eeas
Thus
\beas (q_{0}-1)T(r,f)&\leq& N(r,\frac{1}{M[f]-1})+S^{*}(r,f), \eeas
This completes the proof.
\end{proof}
\begin{proof} [\textbf{Proof of Theorem \ref{th2} }]
In view of Lemma \ref{lem5}, we can write
\bea\label{eq5} &&\mu T(r,f)\\
\nonumber &\leq& \ol{N}(r,\infty;f)+\ol{N}(r,0;f)+(\mu-q_{0})\{N_{k)}(r,0;f)+\ol{N}_{(k+1}(r,0;f)\}\\
\nonumber &+& (\mu_{*}-\mu+q_{0})\ol{N}_{(k+1}(r,0;f)+\ol{N}(r,\frac{1}{M[f]-1})-N_{0}(r,\frac{1}{(M[f])'})+S(r,f).\eea
Thus
\beas (\mu-\mu_{*}-2)T(r,f)&\leq& \ol{N}(r,\frac{1}{M[f]-1})+S(r,f), \eeas
This completes the proof.
\end{proof}
\begin{proof} [\textbf{Proof of Theorem \ref{th3} }]
Since $k\geq1$ and $2\ol{N}_{(k+1}(r,0;f)\leq N(r,0,f)$, so from inequality (\ref{eq5}), we can write
\beas (q_{0}-2) T(r,f) &\leq& \frac{(\mu_{*}-\mu+q_{0})}{2}N(r,0;f)+\ol{N}(r,\frac{1}{M[f]-1})-N_{0}(r,\frac{1}{(M[f])'})+S(r,f).\eeas
Thus \beas (\mu-\mu_{*}-4+q_{0})T(r,f)&\leq& \ol{N}(r,\frac{1}{M[f]-1})+S(r,f)\eeas
This completes the proof.
\end{proof}
\section{Applications}
If there exists positive constants $B_{1},B_{2}>0$ such that
\begin{enumerate}
\item $T(r,f)\leq B_{1}~ N\bigg(r,\frac{1}{M[f]-c}\bigg)+S(r,f),$
\item $T(r,f)\leq B_{2}~\ol{N}\bigg(r,\frac{1}{M[f]-c}\bigg)+S(r,f),$
\end{enumerate}
holds, then we can write
\begin{enumerate}
\item $T(r,M[f])\leq (\mu+\mu_{*})T(r,f)+S(r,f)\leq B_{1}(\mu+\mu_{*})~ N\bigg(r,\frac{1}{M[f]-c}\bigg)+S(r,f),$
\item $T(r,M[f])\leq (\mu+\mu_{*})T(r,f)+S(r,f)\leq B_{2}(\mu+\mu_{*})~\ol{N}\bigg(r,\frac{1}{M[f]-c}\bigg)+S(r,f),$
\end{enumerate}
where $M[f]$ is a differential monomial generated by a non constant transcendental meromorphic function $f$ and $c$ is any non zero constant.\par
Let $\psi=(f)^{q_{0}}(f')^{q_{1}}...(f^{(k)})^{q_{k}}$  and $a$ be a non zero finite value. Then
\bea\label{hbd} \delta(a;\psi)&=&1-\limsup_{r\to\infty}\frac{N(r,a;\psi)}{T(r,\psi)}\\
\nonumber &\leq& 1-\frac{1}{B_{1}(\mu+\mu_{*})}.
\eea
and
\bea\label{hbd1} \Theta(a;\psi)&=&1-\limsup_{r\to\infty}\frac{\ol{N}(r,a;\psi)}{T(r,\psi)}\\
\nonumber &\leq& 1-\frac{1}{B_{2}(\mu+\mu_{*})}.
\eea
Thus the following theorems are immediate in view of Theorems \ref{th1}, \ref{th2}, \ref{th3}.
\begin{theo}\label{th21} Let $f$ be a transcendental meromorphic function and $k\geq2$, $q_{0}\geq2$, $q_{i}\geq0~(i=1,2,..,k-1)$, $q_{k}\geq2$ be integers.
Then
\bea\label{eq1.1} \delta(a;\psi)\leq 1-\frac{q_{0}-1}{(\mu+\mu_{*})}.\eea
\end{theo}
\begin{rem} Thus the Theorem \ref{th21} improves, extends and generalizes the result of Lahiri and Dewan (\cite{ld}).
\end{rem}
\begin{theo}\label{th22}Let $f$ be a transcendental meromorphic function and $k\geq1$, $\mu-\mu_{*}\geq3$, $q_{0}\geq1$, $q_{i}\geq0~(i=1,2,..,k-1)$, $q_{k}\geq1$ be integers.
Then
\bea\label{eq1.2} \Theta(a;\psi)\leq 1-\frac{\mu-\mu^{*}-2}{(\mu+\mu_{*})}.\eea
\end{theo}
\begin{theo}\label{th23}Let $f$ be a transcendental meromorphic function and $k\geq1$, $\mu-\mu_{*}\geq 5-q_{0}$, $q_{0}\geq1$, $q_{i}\geq0~(i=1,2,..,k-1)$, $q_{k}\geq1$ be integers.
Then
\bea\label{eq1.3} \Theta(a;\psi)\leq 1-\frac{\mu-\mu^{*}-4+q_{0}}{(\mu+\mu_{*})}.\eea
\end{theo}


\begin{thebibliography}{99}
\bibitem{8} W. K. Hayman, Meromorphic Functions, The Clarendon Press, Oxford (1964).
\bibitem{hg} X. Huang and Y. Gu, On the value distribution of $f^{2}f^{(k)}$, J. aust. Math. Soc., 78(2005), 17-26.
\bibitem{jh} Y. Jiang and B. Huang, A note on the value distribution of $f^{l}(f^{(k)})^{n}$, Arxiv: 1405.3742v1 [math.CV] 15 May 2014.
\bibitem{ld} I. Lahiri and S. Dewan, Inequalities arising out of the value distribution of a differential monomial, J. Inequal. Pure Appl. Math. 4(2003), no. 2, Article 27.
\bibitem{ly} N. Li and L. Z. Yang, Meromorphic function that shares one small functions with its differential polynomial, Kyungpook Math. J., 50(2010), 447-454.
\bibitem{m} E. Mues, \"{U}ber ein Problem von Hayman, Math. Z.,  164(1979), 239-259.
\bibitem{st} N. Steinmetz, \"{U}ber die Nullstellen von Differential polynomen, Math. Z., 176(1981), 255-265.
\bibitem{yam} K. Yamanoi, Zeros of higher derivatives of meromorphic functions in the complex plane, Proc. Lond. Math. Soc., (3), 106(2013), no. 4, 703-780.
\bibitem{f} C. C. Yang and H. X. Yi, Uniqueness Theory of Meromorphic Functions, Kluwer Academic Publishers.
\bibitem{qz} Q. D. Zhang, A growth theorem for meromorphic functions, J. Chengdu Inst. Meteor., 20, (1992), 12-20.
\end{thebibliography}
\end{document}